\documentclass{article}
\usepackage{amsmath}
\usepackage{amssymb}
\usepackage{amsthm}
\usepackage{graphicx}

\def\L{\ell}

\def\P{\mathbb{P}}
\def\Q{\mathbf{Q}}
\def\R{\mathbf{R}}
\def\I{\mathcal{I}}

\def\U{\mathcal{U}}
\def\V{\mathcal{V}}
\def\W{\mathcal{W}}
\def\p{\pi}

\DeclareMathOperator{\GL}{GL}
\DeclareMathOperator{\OO}{O}
\DeclareMathOperator{\rk}{rank}
\DeclareMathOperator{\Sym}{\mathcal{S}}
\DeclareMathOperator{\diag}{diag}
\DeclareMathOperator{\Span}{span}
\DeclareMathOperator{\Tr}{Tr}
\DeclareMathOperator{\Imag}{Im}

\newtheorem{theorem}{Theorem}
\newtheorem{lemma}{Lemma}
\newtheorem{definition}{Definition}

\begin{document}

\title{At most nine lines in Euclidean three-space \newline have pairwise distance one}

\author{Roland H\"ofer\\ Dr.R.Hoefer@gmail.com}

\date{\today}

\maketitle

\begin{abstract}
J.E.\ Littlewood posed the question of how many infinite circular cylinders of unit radius can be arranged so that each touches all the others.
We give a computer-free proof that one cannot find ten such cylinders.
This improves a known result, namely that there are no eleven such cylinders, which was obtained making partial use of computer verification.
\end{abstract}

\section{Introduction}
J.E.\ Littlewood \cite [p.\ 20]{Littlewood1968} posed the following problem:\\

\fbox{\begin{minipage}{0.9\textwidth}
 {\bf Problem 7}\\
\indent Is it possible in $3$-space for seven infinite circular cylinders of unit radius each
to touch all the others? \hspace{0mm} Seven is the number suggested by counting constants.
\end{minipage}}\\[2mm]

Bozóki, Lee and Rónyai \cite{Bozoki2015} showed that indeed there exist seven mutually touching infinite cylinders.
On the other hand, do there exist eight, nine, ... mutually touching cylinders?
An approach by Kupferberg to construct eight such lines was shown to be invalid in \cite{Ambrus2008}.
Dillon, Koizumi, and Luo \cite{Dillon2025} showed that there do not exist $11$ such mutually touching cylinders making partial use of computer verification, and they gave a computer-free proof that there do not exist $13$ such cylinders.
They use a method based on linear algebra and Ramsey theory.
In the present paper we give a computer-free proof that no ten such cylinders exist, using only linear algebra.

Two infinite cylinders with radius $r>0$ touch if, and only if, the distance between their axes is $2r$.
We restrict to the scaled case $2r=1$, and show that there are no ten Euclidean lines which are pairwise at distance one.

The key point of this paper is the introduction of Veronese-Pl\"ucker coordinates (Definition \ref{DefVPcoord}), which are real, symmetric $6\times 6$ rank-one matrices, as well as the introduction of a nondegenerate bilinear form  $\left<\cdot,\cdot\right>$ on the space of real symmetric $6\times 6$ matrices (Definition \ref{DefBilFo}).
This bilinear form has the remarkable property that, in Veronese-Pl\"ucker coordinates, lines are isotropic (or null) vectors (Lemma \ref{iso}), and two nonparallel lines are at distance one if, and only if, they are orthogonal in Veronese-Pl\"ucker coordinates (Lemma \ref{BilFo}).
It is also remarkable that the Veronese-Pl\"ucker coordinates of any set of more than four lines which are pairwise at distance one are linearly independent (Lemma \ref{lLinIndep}).
In order to show this lemma, we need some preparations: show that the direction vectors of any three distinct lines are linearly independent (Lemma \ref{degLines}), and establish upper bounds of the dimensions of some totally isotropic subspaces (Lemma \ref{index}).
Our main result Theorem \ref{ninelines} is then a direct consequence.

\section{Independence of direction vectors}\label{secPL}

A \emph{line} $\L$ in Euclidean three-space $\R^3$ is a set
\[ \L=\R l + p := \{\lambda l + p \in\R^3\;|\;\lambda\in\R\},\]
where $p,l\in\R^3$ and $l\neq 0$. 
$p$ is a \emph{point} on the line $\L$, and $l$ is a \emph{direction vector} of the line $\L$.
Note that $\L$ possesses two unit direction vectors, namely $\pm l/\sqrt{l\cdot l}$, where $\cdot$ denotes the Euclidean scalar product.

Let $d(p,p')=\sqrt{(p-p')\cdot(p-p')}=\|p-p'\|_2$ denote the Euclidean distance between two points $p,p'\in\R^3$.
The \emph{distance} between two lines $\L$, $\L'$ is the smallest Euclidean distance between any two points on these lines,
\[ d(\L,\L'):=\inf\{d(p,p')\;|\;p\in \L,\,p'\in \L'\}.\]

\begin{lemma}\label{degLines}
a) When considering sets of at least five lines $\L_1,\ldots,\L_n$, $n\ge 5$ of which any two are at distance one, no two of them are parallel. \\
b) If three lines $\L_1,\L_2,\L_3$ have pairwise unit distance and are pairwise nonparallel, then the determinant of their direction vectors is not zero.
\end{lemma}

\begin{proof}
a) Assume $\L_1\parallel\L_2$.
After applying a Euclidean motion, we may assume that $\L_1=\R(1,0,0)^T$, $\L_2=\R(1,0,0)^T+(0,1,0)^T$.

If $\L_i\not\parallel\L_1$, $i=3,4,5$, then applying a Euclidean motion which leaves $\L_1$, $\L_2$ fixed, we may assume that $\L_3=\R(\cos\alpha,\sin\alpha,0)^T+(0,0,z_3)^T$ with $\alpha\in(0,\pi)$.
Clearly, $d(\L_1,\L_3)=d(\L_2,\L_3)=|z_3|$, so $z_3=\pm 1$. 
Choose $z_3=+1$.
$\L_4$ must be of the form  $\R(\cos\beta,\sin\beta,0)^T+(x_4,y_4,z_4)^T$ with $\beta\in(0,\pi)$ and $z_4\in\{-1,+1\}$, cf.\ types 1a and 1b in figure \ref{fig:ParallelLines}.
If $z_4=-1$, then $d(\L_3,\L_4)\ge 2$.
So $z_4=+1$. 
Since any two nonparallel lines in the same plane meet and have distance zero, we must have $\L_4\parallel\L_3$, and $\L_4$ must be one of the lines $\R(\cos\alpha,\sin\alpha,0)^T+(-\sin\alpha,\cos\alpha,1)^T$ and $\R(\cos\alpha,\sin\alpha,0)^T+(\sin\alpha,-\cos\alpha,1)^T$.
But then any line $\L_5$ with $d(\L_i,\L_5)=1$, $i=1,2,3$ must also be one of these two lines, and $d(\L_4,\L_5)\in\{0,2\}$. 
This contradicts our assumption $d(\L_4,\L_5)=1$.

So we may assume that $\L_3\parallel\L_1$ is one of the two lines $\R(1,0,0)^T+\frac{1}{2}(0,1,\pm\sqrt{3})^T$, cf.\ types 2a and 2b in figure \ref{fig:ParallelLines}.
If $\L_4\parallel\L_3$, then $\L_4$ is also of type 2a or 2b, and $d(\L_3,\L_4)\in\{0,\sqrt{3}\}$.
If $\L_4\not\parallel\L_3$, then it is of type 1a or 1b, and $d(\L_3,\L_4)=1\pm\sqrt{3}/2$.
In either case, this contradicts the assumption $d(\L_3,\L_4)=1$.

\begin{figure}[h]
\begin{minipage}[t]{.5\textwidth}
  \centering
  \includegraphics[width=1\linewidth]{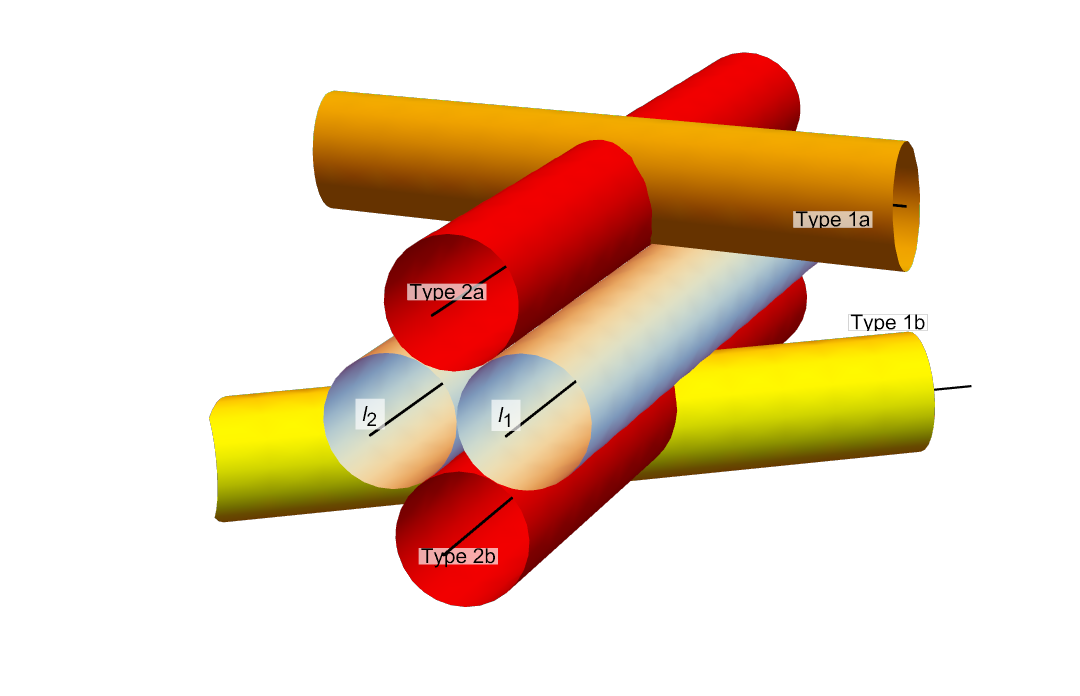}
  \caption{two parallel lines}
  \label{fig:ParallelLines}
\end{minipage}%
\begin{minipage}[t]{.5\textwidth}
  \centering
  \includegraphics[width=0.8\linewidth]{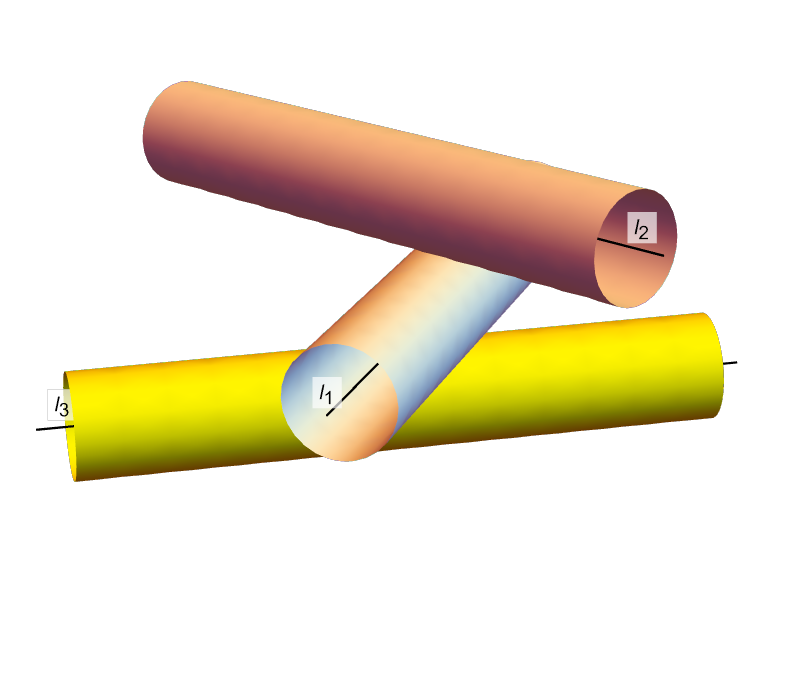}
  \caption{coplanar direction vectors}
  \label{fig:CoplDirVec}
\end{minipage}
\end{figure}

b) Let $l_1,l_2,l_3$ be the direction vectors of $\L_1,\L_2,\L_3$. 
Assume $\det(l_1,l_2,l_3)=0$. 
Then $l_1\times l_2$, $l_1\times l_3$, $l_2\times l_3$ are multiples of a unit vector $u$. 
Let $p_i\in\L_i$ and $d_i:=p_i\cdot u$, $i=1,2,3$.
Then $d(\L_i,\L_j)=|(p_i-p_j)\cdot(l_i\times l_j)|/\|l_i\times l_j\|_2=|(p_i-p_j)\cdot u|=|d_i-d_j|$ for $i,j=1,2,3$.
But $|d_1-d_2| = |d_1-d_3| = |d_2-d_3|=1$ is not possible for real $d_i$. Cf.\ figure \ref{fig:CoplDirVec}.

\end{proof}

\section{Veronese-Pl\"ucker coordinates and a bilinear form}

To the line $\L=\R l + p$ where $p,l\in\R^3$ and $l\neq 0$, we associate the so-called \emph{Pl\"ucker coordinates}, which are the point
\[ \p(\L):=\R(l_1,l_2,l_3,m_1,m_2,m_3)^T \in \P^5(\R)\]
in five-dimensional projective space over $\R$, where the \emph{moment} vector $m\in\R^3$ is defined as
$ m:=p\times l$.
We call any $x\in\p(\L)\setminus\{0\}$ a \emph{representative} of $\p(\L)$.
Note that $\frac{l\times m}{l\cdot l}\in \L$ and that the Pl\"ucker coordinates are independent of the choice of $p\in\L$.
Furthermore, $m\perp l$ implies that the projective point $\p(\L)$ is an element of the \emph{Klein quadric}
\[ \Q:=\{\R(x_1,\ldots,x_6)^T\in\P^5(\R)\;|\;x_1x_4+x_2x_5+x_3x_6=0\}.\]
Conversely, an element $\R(x_1,\ldots,x_6)^T\in\Q$ corresponds to a Euclidean line $\L$ if, and only if, $l:=(x_1,x_2,x_3)^T\neq 0$.
In this case, with $m=(x_4,x_5,x_6)^T$ and $p=\frac{l\times m}{l\cdot l}$, we have $\L= p+\R l=\p^{-1}(\R(x_1,\ldots,x_6)^T)$.
%So, $\p$ takes the set of Euclidean lines in $\R^3$ bijectively to the set $\{\R(x_1,\ldots,x_6)^T\in\Q\;|\;(x_1,x_2,x_3)^T\neq 0\}$.

It is well-known that two lines intersect or are parallel if, and only if, the corresponding points $\R x, \R y\in\Q$ on the Klein quadric are \emph{conjugate}, i.e.,
\[ \left(x,y\right) := x_1 y_4 + x_2 y_5 + x_3 y_6 + x_4 y_1 + x_5 y_2 + x_6 y_3=0.\]
%Two cylinders with radius $1/2$ and axes $\L$ and $\L'$ touch if, and only if, the lines $\L$ and $\L'$ are at distance one.
Let $\R x=\p(\L)$ and $\R x'=\p(\L')$ be the Pl\"ucker coordinates of the two lines $\L$ and $\L'$, and $x=(l_1,l_2,l_3,m_1,m_2,m_3)^T$,
%$x'=(l_1',l_2',l_3',m_1',m_2',m_3')$,
$l=(l_1,l_2,l_3)^T$, $m=(m_1,m_2,m_3)^T$, etc.
If $\L\parallel \L'$, then $l$ and $l'$ are linearly dependent, and we may assume $l=l'$.
Then we have $p:=(l\times m)/(l\cdot l)\in \L$ and $p':=(l\times m')/(l\cdot l)\in \L'$, and $d(\L,\L')=d(p,p')=\|l\times(m-m')\|_2/(l\cdot l)$.
If $\L\nparallel \L'$, then $l\times l'\neq 0$, and a short calculation shows that 
\[ d(\L,\L')=\frac{|\left(x,x'\right)|}{\|l\times l'\|_2}.\]

\begin{definition}\label{DefVPcoord}
For a line $\L$ we define the Veronese-Pl\"ucker coordinates
\[ \varphi(\L) := \binom{l}{m}\binom{l}{m}^T\in\R^{6\times 6}, 
%\left(\begin{array}{ccc} l_1 l_1 & \cdots & l_1 m_3\\ \vdots & \ddots & \vdots\\ m_3 l_1 & \cdots & m_3 m_3\end{array}\right),
\quad\binom{l}{m}\in\pi(\L),\;l\cdot l=1.\]
\end{definition}

Note that the two representatives $\binom{l}{m}$ and $-\binom{l}{m}$ lead to the same Veronese-Pl\"ucker coordinate $\varphi(\L)$, and that $\varphi(\L)$ is a rank-one matrix.

\begin{definition}\label{DefBilFo}
Define a bilinear form $\left<\cdot,\cdot\right>$ on the space of real symmetric $6\times 6$ matrices $\Sym_6(\R)$ as follows:
\[ \left<A,B\right> := \Tr(M_1 A M_1 B) + \Tr(M_2 A M_2 B) - \Tr(M_1 A)\Tr(M_1 B),\]
where
\[ M_1:=\begin{pmatrix} I_3 & 0\\ 0 & 0\end{pmatrix},\quad M_2:=\begin{pmatrix} 0 & I_3\\ I_3 & 0\end{pmatrix}\quad\in\R^{6\times 6}.\]
\end{definition}

\begin{lemma}\label{iso}
The Veronese-Pl\"ucker coordinates $\varphi(\L)$ are isotropic vectors of the space $(\Sym_6(\R),\left<\cdot,\cdot\right>)$, i.e.,
$\left<\varphi(\L),\varphi(\L)\right>=0$ for any line $\L$.
\end{lemma}

\begin{proof}
Let $\L$ be a line with Veronese-Pl\"ucker coordinates $\varphi(\L)=\binom{l}{m}\binom{l}{m}^T$. Then $l^T m=0$ and
$\left<\varphi(\L),\varphi(\L)\right>$ $=$ $(l\cdot l)^2+4(l\cdot m)^2-(l\cdot l)^2$ $=$ $1+0-1$ $=$ $0$.
\end{proof}

\begin{lemma}\label{BilFo}
Two nonparallel lines $\L,\L'$ are at distance one if, and only if, $\left<\varphi(\L),\varphi(\L')\right> = 0$.
\end{lemma}

\begin{proof}
Choose representatives $\binom{l}{m}\in\pi(\L)$, $\binom{l'}{m'}\in\pi(\L')$ with $l\cdot l = l'\cdot l'=1$, then
\begin{eqnarray*}
& & d(\L,\L')=1\\
&\Leftrightarrow & \left |\left({\scriptstyle\binom{l}{m}},{\scriptstyle\binom{l'}{m'}}\right)\right|^2 = \|l\times l'\|_2^2 \\
&\Leftrightarrow & (m\cdot l'+l\cdot m')^2= (l \cdot l) (l' \cdot l')-(l\cdot l')^2\\
&\Leftrightarrow & {\scriptstyle \left(\binom{l}{m}^T M_2 \binom{l'}{m'}\right)^2=\left\|M_1\binom{l}{m}\right\|^2 \left\|M_1\binom{l'}{m'}\right\|^2- \left(\binom{l}{m}^T M_1 \binom{l'}{m'}\right)^2}\\
&\Leftrightarrow & {\scriptstyle \left(\binom{l}{m}^T M_2 \binom{l'}{m'}\right) \left(\binom{l'}{m'}^T M_2 \binom{l}{m}\right) = }\\
& & {\scriptstyle \left( \binom{l}{m}^T M_1\binom{l}{m}\right)  \left(\binom{l'}{m'}^T M_1\binom{l'}{m'}\right) - \left(\binom{l}{m}^T M_1 \binom{l'}{m'}\right)\left(\binom{l'}{m'}^T M_1 \binom{l}{m}\right)}\\
&\Leftrightarrow & \Tr(M_2\varphi(\L) M_2 \varphi(\L')) = \Tr(M_1\varphi(\L))\Tr(M_1\varphi(\L')) - \Tr(M_1\varphi(\L)M_1\varphi(\L'))\\
&\Leftrightarrow & \left<\varphi(\L),\varphi(\L')\right> = 0.
\end{eqnarray*}
In the second to last equivalence we used the cyclicity of trace for products of matrices.
\end{proof}

\section{Results on $n\ge5$ lines}

In the following lemma, we examine properties of $\left<\cdot,\cdot\right>$ when we restrict the Pl\"ucker coordinates to subspaces $\U\subset\R^6$.
We denote the subspace of the corresponding symmetric matrices by $\Sym(\U)$.
More precisely, if $\dim\U=r$ and the columns of $\beta\in\R^{6\times r}$ are a basis of $\U$, then $\Sym(\U):=\beta\Sym_r(\R)\beta^T$.

\begin{lemma}\label{index}
Let $\U\subset\R^6$ be a subspace of dimension $r:=\dim\U$ such that the projection $\U\to\R^3$, $\binom{l}{m}\mapsto l$ is surjective. 
%Let $b_1,\ldots,b_r\in\R^6$ be a basis of $\U$ and define $\Sym(\U):=(b_1,\ldots,b_r)\Sym(r\times r,\R)(b_1,\ldots,b_r)^T$.
Let $\left<\cdot,\cdot\right>_\U$ be the restriction of $\left<\cdot,\cdot\right>$ to $\Sym(\U)\times\Sym(\U)$. 
Then every totally isotropic subspace of $(\Sym(\U),\left<\cdot,\cdot\right>_\U)$ has dimension at most $2r-3\le 9$.
\end{lemma}

\begin{proof}
Since  $\U\to\R^3$, $\binom{l}{m}\mapsto l$ is surjective, choose a basis $\binom{l_1}{m_1},\ldots,\binom{l_r}{m_r}$ of $\U$ such that $(l_1, l_2, l_3)=I_3$ and such that $\binom{l_4}{m_4},\ldots,\binom{l_r}{m_r}$ is a basis of the kernel of $\binom{l}{m}\mapsto l$, i.e.\ $l_4=\ldots,l_r=0$.
Define $L:=(l_1,\ldots,l_r)$, $M:=(m_1,\ldots,m_r)\in\R^{3\times r}$, then
\[ \begin{pmatrix} L \\ M \end{pmatrix} = \begin{pmatrix} I_3 & 0  \\ m_1 \, m_2 \, m_3 & m_4\,\ldots\, m_r\end{pmatrix}\in\R^{6\times r}. \]
Write $(m_1 \, m_2 \, m_3)=N_1\in\R^{3\times 3}$ and $(m_4,\ldots,m_r)=N_2\in\R^{3\times s}$, $s:=r-3$, $\rk(N_2)=s$. 
Choose a matrix $A=\begin{pmatrix} Q & 0\\ B & C\end{pmatrix}\in\R^{r\times r}$, where $Q\in\OO(3,\R)$, $B\in\R^{s\times 3}$, $C\in\GL(s,\R)$.
Then
\begin{eqnarray*}
& & \begin{pmatrix} L \\ M \end{pmatrix} A =  \begin{pmatrix} Q & 0\\ N_1 Q+N_2 B & N_2 C \end{pmatrix},  \\
& & G_1 := A^T \begin{pmatrix} L \\ M \end{pmatrix}^T M_1 \begin{pmatrix} L \\ M \end{pmatrix} A =  \begin{pmatrix} I_3 & 0 \\ 0 & 0\end{pmatrix}\in\R^{r\times r}, \\
& & A^T \begin{pmatrix} L \\ M \end{pmatrix}^T M_2 \begin{pmatrix} L \\ M \end{pmatrix} A =
\begin{pmatrix} Q^T N_1 Q + Q^T N_2 B +Q^T N_1^T Q + B^T N_2^T Q & Q^T N_2 C\\ (Q^T N_2 C)^T & 0\end{pmatrix}. 
\end{eqnarray*}
Choose an orthonormal basis $q_1,\ldots,q_s$ of $\Imag N_2$.
Choose an orthonormal basis $q_{s+1},\ldots,q_3$ of $(\Imag N_2)^\perp$ which diagonalizes the restriction of $N_1+N_1^T$ to $(\Imag N_2)^\perp$.
Define $Q:=(q_1,q_2,q_3)$.
Then  $Q^T N_2=\begin{pmatrix} R \\ 0\end{pmatrix}$ for some $R\in\GL(s,\R)$,  and $Q^T(N_1+N_1^T)Q = \begin{pmatrix} D_{11} & D_{12} \\ D_{12}^T & \Lambda\end{pmatrix}$, where $D_{11}\in\Sym_s(\R)$, $D_{12}\in\R^{s\times t}$, $\Lambda=\diag(\lambda_1,\ldots,\lambda_t)$, $t:=3-s$.
Define $C:=R^{-1}$ and $B:=R^{-1} \begin{pmatrix} -\frac{1}{2}D_{11} & -D_{12}\end{pmatrix} $.
Then we have  $Q^T N_2 C=\begin{pmatrix} I_s \\ 0\end{pmatrix}$ and
\[ Q^T N_2 B = \begin{pmatrix} I_s \\ 0\end{pmatrix}\begin{pmatrix} -\frac{1}{2}D_{11} & -D_{12}\end{pmatrix} = \begin{pmatrix} -\frac{1}{2}D_{11} & - D_{12} \\ 0 & 0 \end{pmatrix}, \]
so
\[ Q^T (N_1+N_1^T) Q  + Q^T N_2 B + (Q^T N_2 B)^T= \begin{pmatrix} 0 & 0 \\ 0 & \Lambda\end{pmatrix},\]
and finally
\[ G_2 := A^T\begin{pmatrix} L \\ M \end{pmatrix}^T M_2 \begin{pmatrix} L \\ M \end{pmatrix} A = \begin{pmatrix} 0 & 0 & I_s \\ 0 & \Lambda & 0 \\ I_s & 0 & 0 \end{pmatrix}\in\R^{r\times r}. \]
So, under the coordinate transformation
\[ \Phi: \Sym_r(\R)\to\Sym(\U), \quad S\mapsto \begin{pmatrix} L \\ M \end{pmatrix} A S A^T \begin{pmatrix} L \\ M \end{pmatrix}^T,\]
the restricted form is
\[ \left<S,T\right>_{\U} = \Tr(G_1 S G_1 T) + \Tr(G_2 S G_2 T) - \Tr(G_1 S) \Tr(G_1 T).\]

In the following we construct a ``{}large''{} positive definite subspace $\W$ of $\Sym_r(\R)$.
First we choose a traceless subspace $\V\subset\Sym_t(\R)$ of dimension $\binom t2$ as follows:
For $t\in\{0,1\}$ take $\V=\{0\}$.  
For $t=2$, take $\V=\R\,\diag(1,-1)$.  
For $t=3$, take a three-dimensional space spanned by the traceless matrices $\diag(1,-1,0)$, $\diag(1,0,-1)$ and $(E_{ij}+E_{ji})$ for a pair $1\le i<j\le 3$ with $1+\lambda_i\lambda_j>0$.  
Such a pair exists since $(\lambda_1\lambda_2)(\lambda_1\lambda_3)(\lambda_2\lambda_3)\ge 0$.  
%Since $\Lambda$ is diagonal, the diagonal and off-diagonal parts are orthogonal for this quadratic form.
For an element $U=\diag(u_1,u_2,u_3)+u_0(E_{ij}+E_{ji})\in\V$ with $u_1+u_2+u_3=0$ we have
\[ \Tr(U^2) + \Tr\bigl((\Lambda U)^2\bigr) = \sum_{k=1}^3 (1+\lambda_k^2) u_k^2+2(1+\lambda_i\lambda_j)u_0^2,\]
so the form $\V\to\R$, $U\longmapsto \Tr(U^2)+\Tr\bigl((\Lambda U)^2\bigr)$ is positive definite.
Now consider the subspace $\W\subset\Sym_r(\R)$ consisting of all matrices
\[ S=\begin{pmatrix}X&Y&Z\\Y^T&U&0\\Z&0&2X\end{pmatrix}\in\Sym_r(\R), \]
where $X\in\Sym_s(\R)$, $Y\in\R^{s\times t}$, $Z=Z^T\in\Sym_s(\R)$, $U\in\V$.
For such an $S$, using $\Tr U=0$, the above formula becomes
\[ \left<S,S\right>_\U=
5\Tr(X^2)-(\Tr X)^2+2\Tr(YY^T)+2\Tr(Z^2)
 +\Tr(U^2)+\Tr\bigl((\Lambda U)^2\bigr). \]
So $\left<S,S\right>_\U$ is positive for every $S\in\W\setminus\{0\}$, since $(\Tr X)^2 = (\Tr(I_s X))^2 \le \Tr(I_s^2)\Tr(X^2) = \Tr(I_s)\Tr(X^2) = s\Tr(X^2) \le 3\Tr(X^2)$ by Cauchy-Schwarz for Frobenius inner product, since $Z$ is symmetric, and since the last two terms are positive definite on $\V$.  
Therefore $\W$ is positive definite.
Its dimension is
\[ \dim\W = s(s+1)+st+\binom{t}{2} %=\binom{r+1}{2}-(2r-3)\\
= \dim\Sym(\U)-(2r-3). \]
If $\I\subset\Sym(\U)$ is totally isotropic, then $\I\cap\Phi(\W)=\{0\}$.  Hence
\[ \dim\I\le \dim\Sym(\U)-\dim\W=2r-3. \]
\end{proof}

\begin{lemma}\label{lLinIndep}
Let $n\ge3$, let $v_i=\binom{l_i}{m_i}\in\R^6$ with $l_i\ne0$, $i=1,\ldots,n$, and set $V_i=v_i v_i^T$, $i=1,\ldots,n$.  
Assume $\left<V_i,V_j\right>=0$ for all $i,j=1,\ldots,n$, and assume that $l_i$, $l_j$, $l_k\in\R^3$ are linearly
independent for all $i\neq j\neq k\neq i$.  
Then the matrices $V_1,\ldots,V_n$ are linearly independent.
\end{lemma}

\begin{proof}
Assume that $\{V_i\}$ are linearly dependent. 
Choose a minimally dependent subset. 
After reindexing, assume it is $V_1,\ldots,V_k$.
Since $l_1,l_2,l_3$ are linearly independent, also $V_1,V_2,V_3$ are linearly independent, and we have $3<k\le n$.
The dimension of $\Span\{V_1,\ldots,V_k\}$ is $k-1$, and there exist $\lambda_i\neq 0$, $i=1,\ldots,k$ such that $\sum_{i=1}^k \lambda_i V_i=0$.
This can be written as  $V\Lambda V^T=0$, where $\Lambda:=\diag(\lambda_1,\ldots,\lambda_k)$ and $V:=(v_1,\ldots,v_k)$.
So the rows of $V$ are isotropic vectors in the vector space $\R^k$ with nondegenerate bilinear form $\R^k\times \R^k\to\R$, $(w,w')\mapsto w\Lambda w'^T$. 
The dimension of the totally isotropic space spanned by the rows of $V$ is $\rk(V)$, thus $2\rk(V)\le k$.
On the other hand, since the projection of $\{v_i\}$ to $\{l_i\}$ spans $\R^3$, we may apply Lemma \ref{index} to $\U:=\Span\{v_1,\ldots,v_k\}$ and the totally isotropic subspace spanned by $V_1,\ldots,V_k$, and obtain $k-1\le 2\dim\U-3=2\rk(V)-3$.
So $2\rk(V)\le k\le 2\rk(V)-2$, a contradiction.
\end{proof}

\begin{theorem}\label{ninelines}
There are at most nine lines in Euclidean three-space which are pairwise at distance one.
\end{theorem}

\begin{proof}
Assume we are given $n$ lines with Pl\"ucker representatives $v_i$ and Veronese-Pl\"ucker coordinates $V_i=v_i v_i^T$, $i=1,\ldots,n$.
For $n\le 4$ there is nothing to prove, so assume $n\ge 5$.
By Lemma \ref{degLines} (a) no two lines are parallel, and by Lemma \ref{degLines} (b) every three direction vectors are independent.
Hence the projection $\Span\{v_1,\ldots,v_n\}\to\R^3$ is surjective.
Lemma \ref{BilFo} and Lemma \ref{iso} give $\left<V_i,V_j\right>=0$ for all $i,j=1,\ldots,n$.
By Lemma \ref{lLinIndep} and Lemma \ref{degLines} the Veronese-Pl\"ucker coordinates $V_1,\ldots,V_n$ of $n\ge 5$ lines must be independent.
So the dimension of the totally isotropic space spanned by $V_1,\ldots,V_n$ is $n$, and $n\le 9$ by Lemma \ref{index}.
\end{proof}

\bibliography{Literatur}
\bibliographystyle{plain}

\end{document}